\documentclass[11pt,a4paper]{article}

\usepackage{soul}
\usepackage[normalem]{ulem}
\usepackage{cancel}

\usepackage{epsf,epsfig,amsfonts,amsgen,amsmath,amstext,amsbsy,amsopn,lineno}
\usepackage{amsmath,amsthm,amsbsy,amssymb}
\usepackage{color}
\usepackage{cite}
\usepackage{subfig}
\usepackage{float}
\usepackage{graphicx,tikz}
\usetikzlibrary{decorations.markings}
\tikzstyle{directed}=[postaction={decorate,
	decoration={markings,mark=at position 0.6 with {\arrow{latex}}}
}]
\usepackage{mathrsfs}
\usepackage[colorlinks=true,citecolor=black,linkcolor=black,urlcolor=black]{hyperref}
\usepackage{hyperref} 
\hypersetup{colorlinks,
    linkcolor=blue, %
    anchorcolor=blue,
    citecolor=blue}
\usepackage{algorithm}

\usepackage{algpseudocode}
\usepackage{chngpage}
\usepackage{mathtools}

\newcommand{\diag}{\mathrm{diag}}

\newtheorem{theorem}{Theorem}[section]
\newtheorem{lemma}[theorem]{Lemma}

\newtheorem{corollary}[theorem]{Corollary}

\newtheorem{proposition}[theorem]{Proposition}
\newtheorem{conjecture}[theorem]{Conjecture}
\newtheorem{example}[theorem]{Example}

\newtheorem{remark}[theorem]{Remark}

\theoremstyle{definition}

\def\Re{\mathbb R}

\def\j{\mbox{\boldmath $j$}}

\def\x{\mbox{\boldmath $x$}}
\def\y{\mbox{\boldmath $y$}}

\def\w{\mbox{\boldmath $w$}}

\def\vec0{\mbox{\boldmath $0$}}

\def\A{\mbox{\boldmath $A$}}
\def\B{\mbox{\boldmath $B$}}
\def\C{\mbox{\boldmath $C$}}
\def\D{\mbox{\boldmath $D$}}

\def\I{\mbox{\boldmath $I$}}

\def\L{\mbox{\boldmath $L$}}
\def\M{\mbox{\boldmath $M$}}

\def\P{\mbox{\boldmath $P$}}
\def\Q{\mbox{\boldmath $Q$}}

\def\Re{\mathbb R}

\tikzstyle{vertex}=[circle, draw, inner sep=0pt, minimum size=3pt]
\newcommand{\vertex}{\node[vertex]}

\numberwithin{equation}{section} 
\allowdisplaybreaks  

\def\qed{\hfill$\Box$\vspace{12pt}}



\begin{document}
\title{The Algebraic Connectivity and Laplacian Spectral Radius of Token Graphs\thanks{This research has been supported by National Natural Science Foundation of China (No.~12471334, No.~12131013).
X. Song is funded by Shenzhen University of Information Technology under project SUIT2026KJ024. 
C. Dalf\'o and M. A. Fiol are funded by AGAUR from the Catalan Government under project 2021SGR00434 and MICINN from the Spanish Government under project PID2020-115442RB-I00.}}
\author{Xiaodi Song$^{a}$, Cristina Dalf\'o$^b$, Miquel \`{A}ngel Fiol$^c$, Shenggui Zhang$^{d}$\\ 
$^a${\small School of Humanities and Fundamental Sciences, Shenzhen University of Information Technology}\\
{\small Shenzhen, Guangdong 518172, P.R. China, {\tt xdsong@suit-sz.edu.cn}}\\
$^b${\small Departament de Matem\`atica, Universitat de Lleida} \\
		{\small Igualada (Barcelona), Catalonia, {\tt{cristina.dalfo@udl.cat}}}\\
$^c${\small Departament de Matem\`atiques, Universitat Polit\`ecnica de Catalunya} \\
    	{\small Barcelona Graduate School, Institut de Matem\`atiques de la UPC-BarcelonaTech (IMTech)}\\
    	{\small Barcelona, Catalonia, {\tt{miguel.angel.fiol@upc.edu}}}\\
$^d${\small School of Mathematics and Statistics, Northwestern Polytechnical University}\\
    {\small Xi'an-Budapest Joint Research Center for Combinatorics, Northwestern Polytechnical University} \\
    {\small Xi'an, Shaanxi, P.R. China, {\tt gzhang@nwpu.edu.cn}}
}
\date{}
\maketitle
\begin{abstract} 
For a graph $G=(V,E)$ of order $n$ and an integer $k$ between $1$ and $\lfloor\frac{n}{2}\rfloor$, its token graph $F_k(G)$ is the graph whose vertices consist of the $\binom{n}{k}$ 
$k$-subsets of $V$, 
and two vertices of $F_k(G)$ are adjacent whenever their symmetric difference is an edge in $E$. It was found that the algebraic connectivity of a graph is greater than or equal to that of its token graph, while the Laplacian spectral radius of a graph is less than or equal to that of its token graph. Moreover, a conjecture that the algebraic connectivity of a graph coincides with that of its token graph has been proved by using the theory of continuous Markov chains of random walks. 

In this paper, we derive some results about the algebraic connectivities of a graph and the same graph after adding new edges and their token graphs to obtain a combinatorial/algebraic proof. Besides, we provide some conditions under which the Laplacian spectral radius of a graph is less than that of its token graph. 
Finally, we characterize the graphs that have the same Laplacian spectral radius as their token graphs, including trees.

\end{abstract}

\noindent{\em Keywords:} Token graph, Laplacian spectrum, Algebraic connectivity, Laplacian spectral radius, Regular partition.
	
\noindent{\em MSC2010:} 05C15, 05C10, 05C50.

\section{Introduction}
Given a graph $G = (V, E)$ of order $n$ and an integer $k$ between $1$ and $\lfloor\frac{n}{2}\rfloor$, the $k$-token graph $F_k(G)$ has vertices representing different configurations of $k$ indistinguishable tokens in different vertices of $G$. Two configurations are adjacent if one can be obtained from the other by moving a token along an edge from its current position to an unoccupied vertex.
Figure \ref{fig:f2p4} shows the $2$-token graph $F_2(P_4)$ of the path of order 4.
The name ``token graphs" originates from Fabila-Monroy, Flores-Peñaloza, Huemer, Hurtado, Urrutia, and Wood\cite{ffhhuw12}. It was also called the {\em symmetric $k$-th power} of a graph in Audenaert, Godsil, Royle, and Rudolph\cite{agrr2007}, and {\em $n$-tuple vertex graphs} in Alavi, Lick, and Liu\cite{all2002}. 

The motivation for researching token graphs stems from their connection to the exchange of Hamiltonian operators in quantum mechanics (see again \cite{agrr2007}) and the graph isomorphism problem. For any two graphs $G$ and $H$, if they are cospectral, do their token graphs have the same spectrum? Conversely, if their token graphs are cospectral, how are the spectra of $G$ and $H$? Rudolph \cite{r2002} showed that there exist some cospectral non-isomorphic graphs that can be distinguished by the adjacency spectrum of their $2$-token graphs, and he also gave an example of the Laplacian spectra. 
For the adjacency spectrum of token graphs, the $2$-token graphs of $G$ and $H$ are cospectral if $G$ and $H$ are cospectral strongly regular graphs with the same parameters (see again \cite{agrr2007}). Besides, Barghi and Ponomarenko\cite{bp2009}, and Alzaga, Iglesias, and Pignol\cite{aip2010} proved that, for each value of $k$, there are infinitely many pairs of non-isomorphic graphs with cospectral $k$-token graphs. What is more, it was found that the adjacency spectrum of a graph is not contained in the spectrum of its token graph (see Rudolph \cite{r2002}).
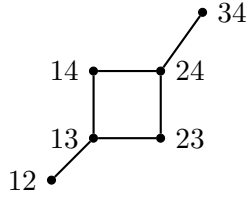
\begin{figure}
\begin{center}
\begin{tikzpicture}[scale=1.11,auto,swap]




\vertex (u13) at (5.5,0) [fill,label=left:{$13$}]{};
\vertex (u12) at (5,-0.5) [fill,label=left:{$12$}]{};
\vertex (u14) at (5.5,0.8) [fill,label=left:{$14$}]{};
\vertex (u23) at (6.3,0) [fill,label=right:{$23$}]{};
\vertex (u24) at (6.3,0.8) [fill,label=right:{$24$}]{};
\vertex (u34) at (6.8,1.5) [fill,label=right:{$34$}]{};

\draw[line width=0.8pt,color=black](u13)--(u14);
\draw[line width=0.8pt,color=black](u14)--(u24);
\draw[line width=0.8pt,color=black](u13)--(u23);
\draw[line width=0.8pt,color=black](u23)--(u24);
\draw[line width=0.8pt,color=black](u12)--(u13);
\draw[line width=0.8pt,color=black](u24)--(u34);
\end{tikzpicture}
\caption{The $2$-token graph of $P_4$ with set of vertices $\{1,2,3,4\}$.}
\label{fig:f2p4}
\end{center}
\end{figure}
In the case of the Laplacian spectrum, Dalf\'o, Duque, Fabila-Monroy, Fiol, Huemer, Trujillo-Negrete, and Zaragoza Mart\'{\i}nez \cite{ddffhtz2021} found that the spectrum of a graph is always contained in the spectrum of its token graphs. 
For example, the spectrum of a complete graph of order $n$ is $\{0,n^{n-1}\}$, while the eigenvalues of its $k$-token graph are $\left\{j(n-j+1):0\le j\le k\right\}$ (see \cite{ddffhtz2021} again). It implies that the algebraic connectivity (the second smallest Laplacian eigenvalue) of a graph is greater than or equal to that of its token graph, while the Laplacian spectral radius (the largest Laplacian eigenvalue) of a graph is less than or equal to that of its token graph. 
This leads to two natural questions:
\begin{itemize}
    \item[$(i)$] When are the algebraic connectivities of $G$ and $F_k(G)$ equal?
    \item[$(ii)$] When are the Laplacian spectral radius of $G$ and $F_k(G)$ equal?
\end{itemize}

In relation to the first question, Dalf\'o, Duque, Fabila-Monroy, Fiol, Huemer, Trujillo-Negrete, and Zaragoza Mart\'{\i}nez \cite{ddffhtz2021} conjectured that the algebraic connectivity of a graph coincides with that of its token graph. 
The conjecture was also known as the {\em Aldous' spectral gap conjecture} \cite{Aldous}. Aldous conjectured that the interchange process and the random walk of unweighted graphs have the same spectral gap, where the (1-particle) random walk on $G$ is the Markov chain in which a single particle jumps from one vertex $u\in V$ to another vertex $v$ through edge $uv\in E$, and the interchange process on $G$ is a Markov chain where a state is an assignment of $k$ labeled particles to the vertices of $G$ (each of the $k$ vertices occupied by exactly one particle), and transitions occur by interchanging the particles at vertices $u$ and $v$ through edge $uv\in E$ to form a new state. 
In fact, the spectral gap of the random walk of an unweighted graph $G$ is exactly the algebraic connectivity of $G$, while that of the $k$-particle interchange process of $G$ is the algebraic connectivity of its token graph $F_k(G)$. 
This conjecture was proved by Caputo, Liggett, and Richthammer\cite{clr10} from a recursive approach based on the idea of network reduction, and they generalized the problem from unweighted graphs to weighted graphs, so that it became the Aldous-CLR theorem. This answers the first question. Their proof is a subtle combination of two ingredients: a nonlinear mapping in the group algebra that allows a proof by induction, and a quite hard estimate known as the Octopus inequality. In 2016, Cesi \cite{Cesi2016} presented a simpler and more transparent version of the Octopus inequality, arising from an algebraic perspective on the Aldous-CLR theorem.

Concerning a combinatorial/algebraic proof, some advances 
have recently been made; see, for instance, Barik and Verma \cite{bv24}, Dalf\'o and Fiol \cite{df2024}, Reyes, Dalf\'o, Fiol, and Messegu\'e \cite{rdfm25}, and Song, Dalf\'o, Fiol, and Zhang\cite{sdfz2024}. 

For the Laplacian spectral radius, Barik and Verma \cite{bv24} showed that the Laplacian spectral radius of a graph $G$ of order $n$ with a vertex of degree $n-1$ is equal to that of its token graph if and only if $G$ is a star. 

In this paper, we continue investigating which classes of graphs we can algebraically prove satisfy the algebraic connectivity theorem, such as graphs with a regular partition. In addition, we also show when the Laplacian spectral radius of a graph is greater than or equal to that of its token graph. 
We prove that among all trees, only star graphs have the same Laplacian spectral radius as their token graphs. Among regular graphs, no graph shares this property. Moreover, the Laplacian spectral radius of complete bipartite graphs, semiregular bipartite graphs, except for the star graphs, is less than that of their token graph. Therefore, we 
present the following conjecture.

\begin{conjecture}
   Let $G$ be a graph of order  $n(\ge 4)$, the equality $\rho(F_k(G))=\rho(G)$ holds for all $k$ with $2\le k\le \lfloor\frac{n}{2}\rfloor$ if and only if $G\cong S_n$. 
\end{conjecture}

This paper is organized as follows. Some basic notations and preliminaries are introduced in the following section. In Section \ref{section:algebraic_connectivity}, we derive some results about the algebraic connectivities of a graph and its token graph, and the same graph after adding some new edges. For instance, we prove that the algebraic connectivity of a complete multipartite graph is the same as that of its token graph. In Section \ref{sec:spactral_radius}, we provide some conditions under which the Laplacian spectral radius of a graph is less than that of its token graph,
and 
We show the cases where the Laplacian spectral radius of a graph coincides with that of its token graph. 

\section{Preliminaries}
Let $\j$ be the all-1 (column) vector, and $\vec0$ be the matrix of appropriate dimension containing only zeros. 
Let $\M\in \Re^{n\times n}$ be a matrix with eigenvalues $\lambda_1(\M)\le \lambda_2(\M)\le\cdots\le \lambda_n(\M)$. 
We define the {\em spectrum} of a matrix $\M$ as the set of eigenvalues together with their multiplicities and denote it by $\text{sp}~\M$.
Let us denote by $\rho(\M)$ the {\em spectral radius} of $\M$, that is, $\rho(\M)=\max\{|\lambda_i(\M)|:i=1,\ldots,n\}.$ 
Denote by $\Phi(\M,x)$ the {\em characteristic polynomial} of $\M$, that is,  $\Phi(\M,x)=\det(x\I-\M)$.
For a matrix $\M\in \mathbb{R}^{n\times n}$, if there exists a permutation matrix ${\Q}\in \mathbb{R}^{n\times n}$ such that
\begin{equation*}
\setlength{\abovedisplayskip}{5pt}
{\Q}^{\top}{\M}{\Q}=\left[{\begin{array}{cc}
\M_{11}&{\M}_{12}\\
\vec0& {\M}_{22}
\end{array}}
\right],
\setlength{\belowdisplayskip}{5pt}
\end{equation*}
where ${\M}_{11}\in \mathbb{R}^{r\times r}$ and ${\M}_{22}\in \mathbb{R}^{(n-r)\times (n-r)}$, then ${\M}$ is \emph{reducible}.
Otherwise, ${\M}$ is an \emph{irreducible matrix}.

Let $G=(V,E)$ be a graph of order  $n$ and $V'\subset V$. 
Two vertices $u,v\in V$ are called {\em adjacent} if $uv\in E$, denoted by $u\sim v$. 
Otherwise, they are {\em non-adjacent}, denoted by $u\nsim v$.

Denote by $G[V']$ the {\em induced subgraph} of $G$, whose vertex set is $V'$ and whose edge set consists of all edges of $G$ with both ends in $V'$, and by $G-V'$ the subgraph of $G$ obtained from $G$ by deleting the vertices together with all the edges incident to the vertices in $V'$. Let $N_G(v)$ be the {\em neighborhood }of $v$, that is, the set of its adjacent vertices. Denote by $d_G(v)$ the {\em degree} of $v$, that is, $d_G(v)=|N_G(v)|$. For simplicity, we write $d_{v}$ for $d_G({v})$. Moreover, we define $d_G(V')$ as the number of edges between $V'$ and $V\setminus V'$. 
A graph $G$ is bipartite if its vertex set $V$ can be partitioned into two disjoint sets $X$ and $Y$ such that every edge of $G$ has one end in $X$ and the other in $Y$. A bipartite graph $G$ is called $(s,t)$-{\em semiregular} if $d_u=s$ for any $u\in X$ and $d_v=t$ for any $v\in Y$.
The {\em Cartesian product} of two graphs $G$ and $H$ is the graph $G~\square~H$ whose vertex set is $V(G)\times V(H)$ such that two vertices $(u_1,v_1), (u_2,v_2)$ are adjacent in $G~\square~H$ if and only if either $u_1u_2\in E(G)$ and $v_1=v_2$, or $u_1=u_2$ and $v_1v_2\in E(H)$.

For a graph $G=(V,E)$ with $V=\{1,2,\ldots,n\}$, let $\D(G)=\diag(d_1,d_2,\ldots,d_n)$. The {\em adjacency matrix} of $G$ is $\A(G)=(a_{ij})\in  \Re^{n\times n}$, where $a_{ij}=1$ if $ij\in E$ and $a_{ij}=0$ otherwise. Let $\L=\L(G)=\D(G)-\A(G)$ be the {\em Laplacian matrix} of $G$. 
Since $\L$ is symmetric positive semidefinite and singular, its eigenvalues are real and non-negative. Denote by $\lambda_1, \lambda_2,\ldots,\lambda_n$, with $(0=)\lambda_1\le \lambda_2\le\cdots\le\lambda_n$, the eigenvalues of $\L$. 
Therefore, the spectral radius of $\L$ is $\rho(\L)=\lambda_n$.
We denote $\rho(G)=\rho(\L)$ as the {\em Laplacian spectral radius} of $G$.
The second smallest eigenvalue $\lambda_2$ is known as the {\em algebraic connectivity} (see Fiedler \cite{f73}), and indicated by $a(G)$.
We define a vector $\x$ as an \emph{embedding} of $G$ if $\x^{\top}\j=0$. For a vertex $v\in V$, the entry of $\x$ corresponding to $v$ is denoted by $\x_v$.
The value $\dfrac{\x^{\top} \L(G)\x}{\x^{\top}\x}$ is known as the \emph{Rayleigh quotient}, and the algebraic connectivity $a(G)$ is such that
\begin{equation}
 a(G)\le \frac{\x^{\top}\L(G)\x}{\x^{\top}\x}=\frac{\sum_{uv\in E(G)}(\x_u-\x_v)^2}{\x^{\top}\x}\label{eq-rayleigh}   
\end{equation}
for 
$\x^{\top}\j=0$ and $\x\ne \vec0$, where equality holds if and only if $\x$ is a {\em Fiedler vector}, that is, the eigenvector of $\L$ corresponding to $a(G)$.

Given a graph $G=(V,E)$ of order $n$ and an integer $k$ with $2\le k\le \lfloor\frac{n}{2}\rfloor$, the {\em $k$-token graph} $F_k(G)$ has vertices representing the $k$-subsets of $V$, and two vertices $A,B\in V(F_k(G))$ are adjacent if their symmetric difference is the two end-vertices of an edge in $E$, that is, $A\bigtriangleup B=\{u,v\}$, where $uv\in E$. Note that $F_k(G)\cong F_{n-k}(G)$, so we only need to consider $1\le k\le \lfloor\frac{n}{2}\rfloor$. Denote by $d_A$ the degree of $A$ in $F_k(G)$. 
The $(n,k)$-\emph{binomial matrix} $\B$ is an $\binom{n}{k}\times n$ matrix whose rows are the characteristic vectors of the $k$-subsets of $[n]$ in a given order. Thus, if the $i$-th $k$-subset is $A$, then
$$
(\B)_{ij}=
\left\lbrace
\begin{array}{ll}
	1 & \mbox{if } j\in A,\\
	0 & \mbox{otherwise.}
\end{array}
\right.
$$
Now, let us describe a pair of results on the interlacing of graph eigenvalues.
\begin{lemma}[Horn and Johnson \cite{Horn_2013}]
\label{le-submatrix}
    Let $\M\in \Re^{n\times n}$ be a Hermitian matrix, partitioned as
    \[
    \M=
\left(
\begin{array}{cccc}
\M_{11} & \M_{12}\\
\M_{21} & \M_{22}
\end{array}
\right),
    \]
    where $\M_{11}\in \Re^{m\times m}$. 
    Then, $\lambda_i(\M)\le \lambda_i(\M_{11})\le 
\lambda_{i+n-m}(\M)$ for $i=1,\ldots,m$.
\end{lemma}
\begin{lemma}[Horn and Johnson \cite{Horn_2013}]
\label{le-matrix_add}
    Let $\M_1$ and $\M_2$ be two Hermitian matrices of order $n$. Then,
    \[
    \lambda_i(\M_1)+\lambda_1(\M_2)\le \lambda_i(\M_1+\M_2)\le \lambda_i(\M_1)+\lambda_n(\M_2)
    \]
    for $i=1,\ldots,n$.
    The equality in the upper bound happens if and only if there is nonzero vector $\x$ such that $\M_1\x=\lambda_i(\M_1)\x$, $\M_2\x=\lambda_n(\M_2)\x$, and $(\M_1+\M_2)\x=\lambda_i(\M_1+\M_2)\x$; equality in the lower bound occurs if and only if there is nonzero vector $\x$ such that $\M_1\x=\lambda_i(\M_1)\x$, $\M_2\x=\lambda_1(\M_2)\x$, and $(\M_1+\M_2)\x=\lambda_i(\M_1+\M_2)\x$.
\end{lemma}
\begin{lemma}[Cvetkovi\'c, Doob, and Sachs \cite{Cvetkovic_1980}]
\label{le-interlacing}
Let $G$ be a graph and $G'=G+uv$, where $uv$ is an edge. Then, the Laplacian eigenvalues of $G$ and $G'$ interlace, that is, 
\[
0=\lambda_1(G)= \lambda_1(G')\le \lambda_2(G)\le \lambda_2(G')\le \cdots\le \lambda_n(G) \le \lambda_{n}(G').
\]
\end{lemma}
\begin{lemma}[Merris \cite{Merris_1998}](Edge Principle)
\label{le-Merris}
Let $G$ be a graph and $G'=G + uv$ or $G'=G - uv$. Let $\lambda$ be a Laplacian eigenvalue of $G$ corresponding to an eigenvector $\x$. If $\x_u=\x_v$, then $\lambda$ is also a Laplacian eigenvalue of $G'$ with eigenvector $\x$, where $G'$ is the graph obtained from $G$ by deleting or adding an edge $uv$ depending on whether or not it is an edge of $G$. 
\end{lemma}

Reasoning similar to the above, we have the following lemma.

\begin{lemma}
\label{le-adding_multi-edge-fiedlervector}
        Let $G$ be a graph, and $G'=G+u_1v_1+\cdots+u_hv_h$ with $h\ge2$. If there exists a Fiedler vector $\x$ of $G$ such that $\x_{u_i}=\x_{v_i}$ for each $i$ with $i=1,\ldots,h$, then $a(G')=a(G)$.
    \end{lemma}
\begin{proof}
    By (\ref{eq-rayleigh}), we have 
    \begin{eqnarray*}
      a(G)=\frac{\x^{\top}\L(G)\x}{\x^{\top}\x}=\frac{\x^{\top}\L(G')\x}{\x^{\top}\x}\ge a(G').\qedhere
    \end{eqnarray*} 
    Together with $a(G')\ge a(G)$ (see Lemma \ref{le-interlacing}), it must be $a(G')= a(G)$. \qed
\end{proof}
Regarding adding more than one edge in $G$, we get the following lemma by directly interlacing Laplacian eigenvalues. 
For completeness, we show the proof.
\begin{lemma}
\label{le-adding_multi-edge}
Let $G=(V,E)$ be a graph of order $n$. If the algebraic connectivity $a(G)$ of $G$ has multiplicity $h(\ge2)$, then $a(G^+)=a(G)$, where $G^+$ is obtained from $G$ by adding any $l$ edges (with $l< h$) that are not in $E$. 
\end{lemma}
\begin{proof}
    Let $e_1,e_2,\ldots,e_{h-1}\notin E$.  
    It follows from Lemma \ref{le-interlacing} that 
    \begin{eqnarray}
        \lambda_2(G) &\le& \lambda_2(G+e_1)\ \le\ \lambda_2(G+e_1+e_2)\ \le\ \cdots \ \le\ \lambda_2(G+e_1+\cdots+e_{h-2}+e_{h-1})\notag\\
        &\le&\lambda_3(G+e_1+\cdots+e_{h-2}) \ \le\  \cdots \ \le\ \lambda_{h}(G+e_1)\ \le \ \lambda_{h+1}(G).\label{eq:interlace}
    \end{eqnarray}
    The inequalities in (\ref{eq:interlace}) hold since $\lambda_i(G+e)\le \lambda_{i+1}(G)$ for $i=1,\ldots,n-1$.
    Since $a(G)$ has multiplicity $h$, we have $\lambda_2(G)=\cdots=\lambda_{h+1}(G)$, which means that all the above expressions hold. 
    Therefore, we obtain $\lambda_2(G)= \lambda_2(G+e_1)=\lambda_2(G+e_1+e_2)=\cdots =\lambda_2(G+e_1+\cdots+e_{h-2}+e_{h-1})$, and the result is satisfied.
    \qed
\end{proof}
Now, let us recall some results related to token graphs.
\begin{lemma}[Dalf\'o, Duque, Fabila-Monroy, Fiol, Huemer, Trujillo-Negrete, and
Zaragoza Mart\'{\i}nez \cite{ddffhtz2021}]
\label{le-Dalfo2021}
Consider a graph $G$ (with $G\cong F_1(G)$) and its $k$-token graph $F_k(G)$, with corresponding Laplacian matrices $\L(G)$ and $\L(F_k(G))$, and the $(n,k)$-binomial matrix $\B$. The following statements hold:
\begin{itemize}
\item[$(i)$] If $\x$ is a $\lambda$-eigenvector of $\L(G)$, then $\B\x$ is a $\lambda$-eigenvector of $\L(F_k(G))$.
\item[$(ii)$] If $\w$ is a $\lambda$-eigenvector of $\L(F_k(G))$ and $\B^{\top}\w\ne 0$, then $\B^{\top}\w$ is a $\lambda$-eigenvector of $\L(G)$.
\end{itemize}
\end{lemma}

Moreover, it is known that the Laplacian spectra of token graphs of $G$ satisfy (see again \cite{ddffhtz2021})
\begin{equation}\label{eq:sp}
    \{0\}=\text{sp}~\L(F_0(G))\subset \text{sp}~\L(F_1(G))\subset\cdots\subset\text{sp}~\L(F_{\lfloor\frac{n}{2}\rfloor}(G)),
\end{equation}
which implies that \begin{equation}
  a(G)\ge a(F_2(G))\ge\cdots\ge a(F_{\lfloor\frac{n}{2}\rfloor}(G)), \label{eq-al-token}
\end{equation}
and 
\begin{equation}
 \rho(F_{\lfloor\frac{n}{2}\rfloor}(G))\ge\cdots\ge \rho(F_2(G))\ge\rho(G).\label{eq-sp-rad-token}
\end{equation}
In this paper, for each $i$ with $1\le i\le \lfloor\frac{n}{2}\rfloor$, $\text{sp}~\L(F_i(G))\backslash \text{sp}~\L(F_{i-1}(G))$ is the multiset consisting of the eigenvalues from $\text{sp}~\L(F_i(G))$ that are not in $\text{sp}~\L(F_{i-1}(G))$, together with their multiplicities.

\section{Some results on the algebraic connectivity of token graphs\label{section:algebraic_connectivity}}
In this section, we investigate the algebraic connectivity of token graphs. Moreover, we focus on the graphs with a regular partition. 
Let $G=(V,E)$ be a graph with vertex set $V=\{1,\ldots,n\}$ and Laplacian matrix $\L$.  
A partition $\pi$ of $V$ into $r$ cells $C_1,C_2,\ldots,C_r$ is called {\em regular} (or {\em equitable}) if, for any pair $i,j\in \{1,\ldots,r\}$ and any vertex $v\in C_i$, the number $b_{ij}=|N_G(v)\cap C_j|$ depends only on $i$ and $j$, that is, every vertex
in $C_i$ has the same number of neighbors in $C_j$ for all $i,j$. 
The {\em quotient graph $G/\pi$} of $G$ with respect to the partition $\pi$ is a weighted digraph with vertex set $\{C_1,\ldots,C_r\}$, and there is an arc with weight $b_{ij}$ from vertex $C_i$ to $C_j$ if and only if $b_{ij}\ne0$. 
Consider the circulant graph $C(7;1,4)$ (see the right in Figure \ref{fig:quotient}) with vertex set $\{0,1,\ldots,6\}$, and the regular partition $\pi=(C_1,C_2,C_3,C_4)$ with $C_1=\{0\}, C_2=\{1,6\}, C_3=\{2,5\}$ and $C_4=\{3,4\}$. Its quotient graph is shown on the left in Figure \ref{fig:quotient}.
We define the {\em quotient Laplacian matrix $\Q_L=\L(G/\pi)$} of $G$ in relation to $\pi$ as the $r\times r$ matrix such that $$
(\Q_L)_{ij}=\left\{
 \begin{array}{lll}
 -b_{ij}, & \mbox{if } i\ne j;\\
 \displaystyle\sum_{l=1}^r b_{il}-b_{ii}, & \mbox{otherwise}.
 \end{array}
 \right.
$$
For more information on this matrix, see Cardoso, Delorme, and  Rama \cite{cdr2007}, and Reyes, Dalf\'o, and Fiol \cite{rdf24}.


First, let us show some results concerning the algebraic connectivity of a graph after adding a new edge to it.
\begin{theorem}[Song, Dalf\'o, Fiol, and Zhang \cite{sdfz2024}]
\label{th-add_edge}
Let $G=(V,E)$ be a graph of order $n$ such that, for some $k$ with $1\le k\le \lfloor\frac{n}{2}\rfloor$, we have $a(F_k(G))=a(G)$. Consider adding 
an edge
$uv\notin E$, 
getting the new graph $G+uv$. 
If $a(G+uv)=a(G)$, then $a(F_k(G+uv))=a(G+uv)$.
\end{theorem}
We get the following result directly by combining the above theorem and Lemma \ref{le-adding_multi-edge}.
\begin{theorem}
\label{th-add-multi-edge}
    Let $G=(V,E)$ be a graph of order $n$ such that, for some $k$ with $1\le k\le \lfloor\frac{n}{2}\rfloor$, we have $a(F_k(G))=a(G)$. Suppose that the algebraic connectivity $a(G)$ of $G$ has multiplicity $h(\ge2)$. Let $G^+$ be the graph obtained from $G$ by adding any $l$ (with $l< h$) edges that are not in $E$. Then
    $$
    a(F_k(G^+))=a(G^+)=a(G)=a(F_k(G)).
    $$
\end{theorem}

Next, we have the following result on adding edges in a graph with a regular partition.
\begin{theorem}
\label{the-regular}
Let $G$ be a graph with regular partition $\pi$ having cells $C_1,C_2,\ldots,C_r$ such that, for some $k$ with $1\le k\le \lfloor\frac{n}{2}\rfloor$, we have $a(F_k(G))=a(G)$. Let $\Q_{L}$ be the Laplacian matrix quotient of the partition. Assume that the algebraic connectivity $a(G)$ of $G$ is an eigenvalue of $\Q_{L}$. Let $G^+$ be the graph obtained from $G$ by adding any number of edges in $C_i$, for some $i=1,\ldots,r$. Then $a(F_k(G^+))=a(G^+)=a(G)$.
\end{theorem}
\begin{proof}
Recall that by the Aldous-CLR theorem, the equality $a(F_{k}(G)) = a(G)$ holds immediately, but we investigate for which classes of graphs this theorem can be proved algebraically. 

Denote by $\P\in \mathbb{R}^{n\times r}$ the characteristic matrix of $\pi$ whose columns are the characteristic vectors of the cells $C_1,\ldots,C_r$ of $\pi$. That is, for $i\in V$ and $j\in\{1,\ldots,r\}$, $(\P)_{ij}=1$ if $i\in C_j$, and $(\P)_{ij}=0$, otherwise. Then, we have the following equality
$$
\L\P=\P\Q_L.
$$
As a consequence, if $\x$ is an eigenvector of $\Q_L$ with eigenvalue $\lambda$, then $\y=\P\x$ is an eigenvector of $\L$ with eigenvalue $\lambda$  (see again \cite{cdr2007} and \cite{rdf24}). 
Moreover, note that
$$
\y_i=(\P\x)_i=\sum_{l=1}^r (\P)_{il}\x_{l}=\x_{j}
$$
for every $i\in C_j$. 
Thus, for each eigenvalue $\lambda$ of $\Q_L$, we have a $\lambda$-eigenvector of $\L$ with constant entries corresponding to the vertices of each $C_i$.

    Since $a(G)$ is an eigenvalue of $\Q_L$ with the eigenvector $\x$, we obtain that $\y=\P\x$ is a Fiedler vector of $\L$ corresponding to $a(G)$ such that $\y_{i_1}=\y_{i_2}$ for each $i_1,i_2\in C_j$, where $j=1,\ldots,r$. It follows from Lemma \ref{le-adding_multi-edge-fiedlervector} that $a(G)=a(G^+)$. Together with Theorem \ref{th-add_edge}, it must be $a(F_k(G^+))=a(G^+)=a(G)$.\qed
\end{proof}

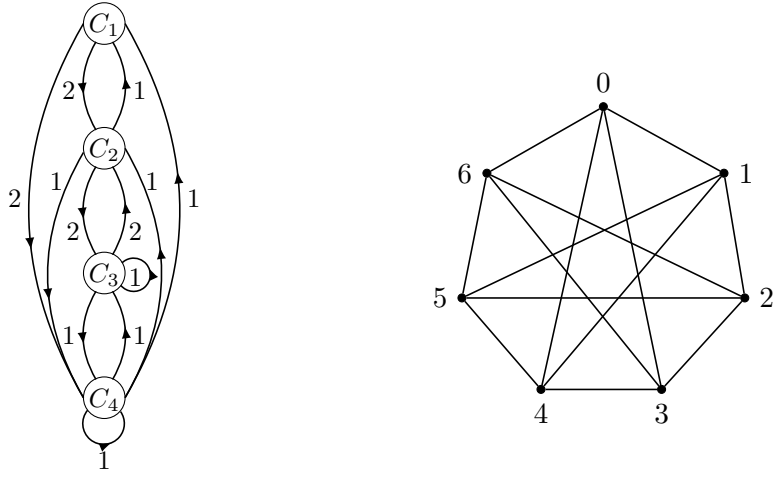
\begin{figure}   
    \begin{center}
    \begin{tikzpicture}[scale=1.1,auto,swap]
    \draw[] (1,1) circle (0.25);
    \node () at (1,0.6) [label=above:{{\small $C_1$}}]{};
    \draw[] (1,-0.5) circle (0.25);
    \node () at (1,-0.9) [label=above:{{\small $C_2$}}]{};
    \draw[] (1,-2) circle (0.25);
    \node () at (1,-2.4) [label=above:{{\small $C_3$}}]{};
    \draw[] (1,-3.5) circle (0.25);
    \node () at (1,-3.9) [label=above:{{\small $C_4$}}]{};
    \draw (0.9,0.78) edge [directed,line width=0.6pt,bend right] (0.9,-0.28);
    \draw (1.1,-0.28) edge [directed,line width=0.6pt,bend right] (1.1,0.78);
    \draw (0.9,-0.72) edge [directed,line width=0.6pt,bend right] (0.9,-1.78);
    \draw (1.1,-1.78) edge [directed,line width=0.6pt,bend right] (1.1,-0.72);
    \draw (0.9,-2.22) edge [directed,line width=0.6pt,bend right] (0.9,-3.28);
    \draw (1.1,-3.28) edge [directed,line width=0.6pt,bend right] (1.1,-2.22);
    \draw (0.75,-0.55) edge [directed,line width=0.6pt,bend right] (0.75,-3.5);
    \draw (1.25,-3.5) edge [directed,line width=0.6pt,bend right] (1.25,-0.5);
    \draw (0.75,1) edge [directed,line width=0.6pt,bend right] (0.75,-3.5);
    \draw (1.25,-3.5) edge [directed,line width=0.6pt,bend right] (1.25,1);
    \draw [directed,line width=0.6pt,bend right] (0.8,-3.65) arc (140:395:0.25);
    \draw [directed,line width=0.6pt,bend right] (1.2,-2.15) arc (220:495:0.2);
    \node () at (0.9,0.2) [label=left:{\small 2}]{};
    \node () at (1.1,0.2) [label=right:{\small 1}]{};
    \node () at (0.95,-1.5) [label=left:{\small 2}]{};
    \node () at (1.05,-1.5) [label=right:{\small 2}]{};
    \node () at (0.9,-2.75) [label=left:{\small 1}]{};
    \node () at (1.1,-2.75) [label=right:{\small 1}]{};
    \node () at (1,-3.9) [label=below:{\small 1}]{};
    \node () at (0.25,-1.1) [label=left:{\small 2}]{};
    \node () at (1.75,-1.1) [label=right:{\small 1}]{};
    \node () at (0.75,-0.9) [label=left:{\small 1}]{};
    \node () at (1.25,-0.9) [label=right:{\small 1}]{};
    \node () at (1.05,-2.05) [label=right:{\small 1}]{};
    \vertex (0) at (7,0) [fill,label=above:0]{};
    \vertex (4) at (6.25,-3.4) [fill,label=below:4]{};
    \vertex (3) at (7.7,-3.4) [fill,label=below:3]{};
    \vertex (2) at (8.7,-2.3) [fill,label=right:2]{};
    \vertex (1) at (8.45,-0.8) [fill,label=right:1]{};
    \vertex (5) at (5.3,-2.3) [fill,label=left:5]{};
    \vertex (6) at (5.6,-0.8) [fill,label=left:6]{};
    \draw[line width=0.6pt](0)--(1);
    \draw[line width=0.6pt](1)--(2);
    \draw[line width=0.6pt](2)--(3);
    \draw[line width=0.6pt](3)--(4);
    \draw[line width=0.6pt](4)--(5);
    \draw[line width=0.6pt](5)--(6);
    \draw[line width=0.6pt](6)--(0);
    \draw[line width=0.6pt](0)--(4);
    \draw[line width=0.6pt](4)--(1);
    \draw[line width=0.6pt](5)--(1);
    \draw[line width=0.6pt](5)--(2);
    \draw[line width=0.6pt](2)--(6);
    \draw[line width=0.6pt](6)--(3);
    \draw[line width=0.6pt](3)--(0);
    \end{tikzpicture}
    \caption{The quotient graph of a regular partition of the circulant graph $C(7;1,4)$ (left) and $C(7;1,4)$ (right). 
    }
     \label{fig:quotient}
    \end{center}
\end{figure}

What about adding edges to a graph with a regular partition such that
the algebraic connectivity $a(G)$ of $G$ is not an eigenvalue of its Laplacian quotient matrix $\Q_{L}$?
In a {\em multipartite graph}, its vertex set can be partitioned into $r$ parts $C_1,\ldots,C_r$ such that no edge has both ends in the same part. A multipartite graph is {\em complete} if any two vertices in different parts are adjacent. 
Let $\pi$ be a partition of a complete multipartite graph $K_{n_1,n_2,\ldots,n_r}$ with $n=\sum_{i=1}^rn_i$ such that the parts of $V(K_{n_1,n_2,\ldots,n_r})$ are the cells of $\pi$. 
Note that $\pi$ is a regular partition. Without loss of generality, suppose that $n_1\le n_2\le \cdots\le n_r$. 
Zhao and Wang \cite{zw2011} showed that the Laplacian characteristic polynomial of a complete multipartite graph is  
\begin{equation}
    \Phi(\L,x)=x(x-n)^{r-1}\prod_{i=1}^r(x-n+n_i)^{n_i-1}.\label{eq-laplacian-poly-multipartite}
\end{equation}
Therefore, we have $a(K_{n_1,n_2,\ldots,n_r})=n-n_r$. Observe that the quotient Laplacian matrix $\Q_{L}=K_{n_1,n_2,\ldots,n_r}/\pi$ with respect to the partition $\pi$ is 
$$
\Q_L=
\left(
\begin{array}{cccc}
n-n_1 & -n_2 & \ldots & -n_r\\
-n_1 & n-n_2 & \ldots & -n_r\\
\vdots & \vdots & \ddots & \vdots\\
-n_1 & -n_2 & \ldots &n-n_r
\end{array}
\right)
$$
and the spectrum of $\Q_L$ is $\{0,n^{r-1}\}.$ Therefore, the algebraic connectivity of $K_{n_1,n_2,\ldots,n_r}$ is not an eigenvalue of $\Q_{L}$. 

Before investigating how the algebraic connectivity of a token graph of a complete multipartite graph changes by adding some edges to it, we need the following results.
Given a graph $G$ of order $n$ and an integer $k$ with $1\le k\le \lfloor\frac{n}{2}\rfloor$, the {\em $k$-algebraic connectivity} $a_k(G)$ and the {\em $k$-spectral radius }$\rho_k(G)$ of $G$ are the minimum and maximum eigenvalues of the multiset $\text{sp}~\L(F_k(G))\backslash \text{sp}~\L(F_{k-1}(G))$. For more information on $a_k(G)$ and $\rho_k(G)$, see Reyes, Dalf\'o, and Fiol \cite{rdf24}.
For a vertex $v\in V(G)$, let $S_v:=\{A\in V(F_k(G)):v\in A\}$ and $S'_v:=\{B\in V(F_k(G)):v\notin B\}$. Let $H_v$ and $H'_v$ be the subgraphs of $F_k(G)$ induced by $S_v$ and $S'_v$, respectively. Note that $H_v\cong F_{k-1}(G-v)$ and $H'_v\cong F_k(G-v)$.
\begin{lemma}[Dalf\'o and Fiol \cite{df2024}]
    \label{le-embedding}
    Given a vertex $v\in V(G)$ and an eigenvector $\w$ of $F_k(G)$ such that $\B^{\top}\w=0$, let $\y_v:=\w|_{S_v}$ and $\y_v':=\w|_{S_v'}$.
    Then, $\y_v$ and $\y_v'$ are embeddings of $H_v$ and $H_v'$, respectively.
\end{lemma}
\begin{theorem}[Song, Dalf\'o, Fiol, and Zhang \cite{sdfz2024}]
\label{th-clique}
Let $G$ be a graph of order $n$, with a cut clique $K_r$ such that $G-K_r$ results in some disjoint graphs $G_1,\ldots,G_s$, for $s\ge 2$, and 
 degree $d_G(K_r)=r(n-r)$ (so that there are all the edges between the vertices of $K_r$ and the vertices of each $G_i$).  Then, $a(G)=a(F_k(G))=r$ for each $k=1,\ldots,\lfloor\frac{n}{2}\rfloor$.
\end{theorem}
\begin{lemma}[Dalf\'o and Fiol \cite{df2024}, Reyes, Dalf\'o, and Fiol \cite{rdf24}]
    \label{le-k-algebraic}
    For a complete graph $K_n$ with order $n$, the distinct eigenvalues of $F_k(K_n)$ are 
    \begin{equation*}
    0,n,2(n-1),\ldots,k(n-k+1),
    \end{equation*}
    and $m_j=\binom{n}{j}-\binom{n}{j-1}$ is the multiplicity of the eigenvalue $j(n-j+1)$ for $j=0,\ldots,k$. 
    
    Moreover, let $\lambda^{j}_1\le \lambda^{j}_2\le\cdots\le\lambda^{j}_{m_j}$ be the eigenvalues of $\text{sp}~\L(F_j(G))\backslash \text{sp}~\L(F_{j-1}(G))$ in non-decreasing order, and $\overline{\lambda}^{j}_1\ge \overline{\lambda}^{j}_2\ge\cdots\ge\overline{\lambda}^{j}_{m_j}$ the eigenvalues of $\text{sp}~\L(F_j(\overline{G}))\backslash \text{sp}~\L(F_{j-1}(\overline{G}))$ in non-increasing order, where $j=1,\ldots,k$. Then, $$\lambda^{j}_i+\overline{\lambda}^{j}_{i}=j(n-j+1)$$
    for $i=1,\ldots,m_j$.
    Let $i=1$ and $j=k$. It implies that the $k$-algebraic connectivity of $G$ and $k$-spectral radius of $\overline{G}$ satisfy
    \[a_k(G)+\rho_k(\overline{G})=k(n-k+1).\]
\end{lemma}
\begin{lemma}[Lew \cite{Lew2024}]
\label{le-bound-k-algebraic}
    Let $G$ be a graph of order  $n$, and let $1\le k\le \lfloor\frac{n}{2}\rfloor$. Suppose that $\mu_k(G)$ is an eigenvalue of {\em sp}\,$\L(F_k(G))\backslash$ {\em sp}\,$\L(F_{k-1}(G))$. Then, \[k(a(G)-k+1)\le\mu_k(G)\le k\rho(G).\]
\end{lemma}
\begin{lemma}
\label{le-value}
    For three integers $n,k,b$ such that $1\le k\le  \lfloor\frac{n}{2}\rfloor$ and $b\le \lfloor\frac{n}{2}\rfloor$, we have \[kb\le k(n-k+1)-(n-b).\]
\end{lemma}
\begin{proof}
    It suffices to prove $k(n-k+1-b)\ge n-b$. Note that $b\le \lfloor\frac{n}{2}\rfloor$. We have $n-b\ge n-\lfloor\frac{n}{2}\rfloor\ge \lfloor\frac{n}{2}\rfloor\ge k$, that is, $n-b-k\ge 0$. Then, it must be $k(n-b-k)\ge n-b-k$, as desired.\qed
\end{proof}

Now, we present a result on complete multipartite graphs by an algebraic proof.

\begin{theorem}
\label{th-token-com-mul}
Let $G=K_{n_1,n_2,\ldots,n_r}$ be a complete multipartite graph with $n_1\le n_2\le \cdots\le n_r$ and $n=\sum_{i=1}^rn_i$. Then, $a(G)=a(F_k(G))=n-n_r$ for each integer $k$ with $1\le k\le \lfloor\frac{n}{2}\rfloor$.  
\end{theorem}
\begin{proof}
    Recall that since $a(F_k(G))\le a(G)=n-n_r$, it suffices to prove $a(F_k(G))\ge n-n_r$.
    Let $\pi=(C_1,\ldots,C_r)$ be a regular partition of $G$ such that the parts of $V$ are the cells of $\pi$. 
    Without loss of generality, assume that $n_r\ge2$. Otherwise, $n_r=1$ implies that $n_1=\cdots=n_r=1$, and $G$ is a complete graph.
    Suppose first that $n_r\le \lfloor\frac{n}{2}\rfloor$.
    Consider $\overline{G}=K_{n_1}\cup K_{n_2}\cup \cdots\cup K_{n_r}$.
    Note that $\rho_k(\overline{G})$ is an eigenvalue of $\text{sp}~\L(F_k(\overline{G}))\backslash \text{sp}~\L(F_{k-1}(\overline{G}))$, then,  from Lemmas \ref{le-bound-k-algebraic} and \ref{le-value}, it must be \begin{eqnarray}
        \rho_k(\overline{G})\le k\rho(\overline{G})&=&kn_r\notag\\
        &\le& k(n-k+1)-(n-n_r).\label{eq:rhok}
    \end{eqnarray}
    Observe that $a(F_k(G))=\min\{a(G),a_2(G),\ldots,a_k(G)\}$ (see the definition of $a_k(G)$ in the previous page). Then, it follows from Lemma \ref{le-k-algebraic} that $a_k(G)= k(n-k+1)-\rho_k(\overline{G})\ge n-n_r$ for each $k$ with $1\le k\le \lfloor\frac{n}{2}\rfloor$. It implies that $a(F_k(G))\ge n-n_r$.

    Next, assume that $n_r>\lfloor\frac{n}{2}\rfloor$. Then, it must be $n_r>n_i$ for each $i=1,\ldots,r-1$, which implies that $n_r-1\ge n_{r-1}$. Let $\x$ be a Fiedler vector of $F_k(G)$. Suppose that $\x^{\top}\x=1$. 
    From Lemma \ref{le-Dalfo2021}, if $\B^{\top}\x\ne0$, then $a(F_k(G))$ is an eigenvalue of $\L$, which implies that $a(F_k(G))\ge a(G)=n-n_r$, so it suffices to consider $\B^{\top}\x=0$.
    We prove it by induction on $n$ (with $n\ge r+1$) and $k$. For $k=1$, $a(F_1(G))=a(G)=n-n_r$. For $n=r+1$, we obtain $n_1=\cdots=n_{r-1}=1<2=n_r$. Note that $G[C_1\cup \cdots\cup C_{r-1}]$ is a cut clique $K_{r-1}$ such that $G-(C_1\cup \cdots\cup C_{r-1})$ results in two isolated vertex, and $d_G(C_1\cup \cdots \cup C_{r-1})=d_G(K_{r-1})=(r-1)(n-(r-1))=2(r-1)$. 
    It follows from Theorem \ref{th-clique} that $a(F_k(G))=r-1=n-n_r$ for any $k$. 
    Suppose next that $n\ge r+2$ and $k\ge 2$. 
    Let $v$ be a vertex in $C_r$. Thus, $G-v\cong K_{n_1,\ldots,n_{r-1},(n_r)-1}$ is also a complete multipartite graph with order $n-1$. Note that $\x$ is a Fiedler vector of $F_k(G)$. Let $\y_v:=\x|_{S_v}$ and $\y'_v:=\x|_{S'_v}$ such that, by Lemma \ref{le-embedding}, are embeddings of $H_v\cong F_{k-1}(G-v)$ and $H'_v\cong F_k(G-v)$, respectively. Then, from (\ref{eq-rayleigh}), and the hypothesis, we obtain \[
    \frac{\sum_{(A,B)\in E(H_v)}\left((\y_v)_A-(\y_v)_B\right)^2}{\sum_{A\in V(H_v)}((\y_v)_A)^2}\ge a(H_v)=a(F_{k-1}(G-v))\ge (n-1)-(n_r-1)=n-n_r
    \] 
    and 
    \[
    \frac{\sum_{(A,B)\in E(H_v')}((\y_v')_A-(\y_v')_B)^2}{\sum_{A\in V(H_v')}((\y_v')_A)^2}\ge a(H_v')=a(F_k(G-v))\ge n-n_r.
    \]
Thus, 
    \begin{eqnarray*}
        a(F_{k}(G))&=&\x^{\top}\L(F_k(G))\x\\
        &=&\sum_{(A,B)\in E(F_k(G))}(\x_A-\x_B)^2\\
        &\ge &\sum_{(A,B)\in E(H_v)}\left((\y_v)_A-(\y_v)_B\right)^2+\sum_{(A,B)\in E(H_v')}\left((\y_v')_A-(\y_v')_B\right)^2\\
        &\ge &a(F_{k-1}(G-v))\sum_{A\in V(H_v)}\left((\y_v)_A\right)^2+a(F_k(G-v))\sum_{A\in V(H_v')}\left((\y_v')_A\right)^2\\
        &\ge &(n-n_r)\left[\sum_{A\in V(H_v)}(\x_A)^2+\sum_{A\in V(H_v')}(\x_A)^2\right]\\
        &=&n-n_r. 
    \end{eqnarray*}
    This completes the proof.\qed
    \end{proof}
\begin{remark}
    Dalf\'o, Duque, Fabila-Monroy, Fiol, Huemer, Trujillo-Negrete, and Zaragoza Mart\'{\i}nez \cite{ddffhtz2021} proved that the algebraic connectivity of a complete bipartite graph coincides with that of its token graph, see \cite{ddffhtz2021} Theorem 7.2 (iv). However, the proof is incomplete. Let $G=K_{n_1,n_2}$ be a complete bipartite graph. 
    From Lemma \ref{le-k-algebraic}, to obtain the algebraic connectivity of $F_k(G)$, it is necessary to consider all possible cases $\lambda(F_k(K_n))-\lambda(F_k(\overline{G}))$, where $\lambda(F_k(K_n))$ is one of the eigenvalues of $F_k(K_n)$ and $\lambda(F_k(\overline{G}))$ is one of the eigenvalues of $F_k(\overline{G})$, respectively. 
    Note that for each $k_1,k_2$ that satisfy $k_1+k_2=k$, $F_k(\overline{G})$ has a component isomorphic to $F_{k_1}(K_{n_1})~\square~F_{k_2}(K_{n_2})$ (See the proof of \cite{ddffhtz2021} Corollary 6.4). Then, the spectrum of $F_k(\overline{G})$ is the union of the spectra of each component $F_{k_1}(K_{n_1})~\square~F_{k_2}(K_{n_2})$ with $k_1+k_2=k$.
    Then, the distinct eigenvalues of $F_k(\overline{G})$ should be 
\begin{align*}
    \big\{&0,n_1,2(n_1-1),\ldots,k_1(n_1-k_1+1),\\
    &n_2,n_2+n_1,n_2+2(n_1-1),\ldots,n_2+k_1(n_1-k_1+1),\\
    &2(n_2-1),2(n_2-1)+n_1,2(n_2-1)+2(n_1-1),\ldots,2(n_2-1)+k_1(n_1-k_1+1),\\
    &\cdots\\
    &k_2(n_2-k_2+1),k_2(n_2-k_2+1)+n_1,k_2(n_2-k_2+1)+2(n_1-1),\ldots,\\
    &k_2(n_2-k_2+1)+k_1(n_1-k_1+1)\big\}.
\end{align*}
It follows from Lemma \ref{le-k-algebraic} that the distinct eigenvalues of $\lambda(F_k(K_n))$ are $0,n,2(n-1),\ldots,k(n-k+1)$. 
Each case should be more detailed in \cite{ddffhtz2021} Theorem 7.2 (iv).

In addition, note that Theorem \ref{th-token-com-mul} implies that the algebraic connectivity of a complete bipartite graph coincides with that of its token graph. 
Suppose that $r=2$ in Theorem \ref{th-token-com-mul}. Then, we have $a(K_{n_1,n_2})$ \ $=a(F_k(K_{n_1,n_2}))=n-n_2=n_1$ for each integer $k$ with $1\le k\le \lfloor\frac{n}{2}\rfloor$.
\end{remark}

Now, we consider how the algebraic connectivity changes after adding edges to a complete multipartite graph. Adding some edges within some $C_i$ (for $i=1,\ldots,r$) in $K_{n_1,n_2,\ldots,n_r}$, we get the following results.
\begin{proposition}
    Let $G^+$ be a graph obtained from $G=K_{n_1,n_2,\ldots,n_r} (with \sum_{i=1}^rn_i=n)$ by adding some edges within some $C_i$ for $i=1,\ldots,r$, such that the graph $G^+[C_j]$ induced by the vertices in $C_j$ is connected, for every $C_j$ with $|C_j|=n_r$. Then, $a(G^+)>a(G)=n-n_r$. 
\end{proposition}
\begin{proof}
    Suppose that there exists an integer $s$ (with $1\le s\le r$) such that $n_1\le\cdots\le n_{s-1}<n_s=\cdots=n_r$. 
    Note that $\L(G^+)=\L(G)+\M$, where $\M$ is the Laplacian matrix of the graph induced by the edges $E(G^+)\backslash E(G)$.
    Let $H_1,H_2,\ldots,H_t$ be the connected components of the graph induced by the edges $E(G^+)\backslash E(G)$. 
    Then $\M$ is the block diagonal matrix whose diagonal blocks are $\L(H_1),\L(H_2),\ldots,\L(H_t)$.  
    Clearly, the smallest eigenvalue of $\M$ is $\lambda_1(\M)=0$. 
    By Lemma \ref{le-matrix_add}, it follows that $a(G^+)\ge a(G)+\lambda_1(\M)$. 
    If the equality holds, then there is a vector $\x$ such that $\L(G)\x=a(G)\x$ and $\M\x=\lambda_1(\M)\x$. 
    Note that the eigenvector $\x$ corresponding to $\lambda_1(\M)$ is $(\vec0^{\top},\y_i^{\top},\vec0^{\top})^{\top}$ such that $\y_i$ is the eigenvector of $\L(H_i)$ corresponding to eigenvalue $0$.
    Since $G^+[C_j]$ is connected for $j=s,\ldots,r$, there exists a Fiedler vector $\x$ of $G$ such that $\x_u=\x_v$ for every pair of vertices $u,v\in C_j$.
    %
    Let $G^*$ be a graph obtained from $G$ by adding all edges in each $C_j$ such that $G^*[C_j]$ is a complete graph of order $n_r$, where $j=s,\ldots,r$. Thus, we have $a(G^*)=a(G)=n-n_r$. Suppose first that $2\le s\le r$. Note that $G^*$ is also a complete multipartite graph in which the maximum order among its parts is $n_{s-1}<n_s=n_r$. It implies $a(G^*)=n-n_{s-1}>n-n_r$, which is a contradiction. Next, assume that $s=1$, which implies that $n_1=\cdots=n_r$. 
    Then, we get that $G^*$ is a complete graph with $a(G^*)=n\ne n-n_r$, a contradiction. Therefore, we obtain $a(G^+)>a(G)=n-n_r$.\qed
\end{proof}
\begin{theorem}
\label{th-al-com-mul}
Let $G^+$ be a graph obtained from $G=K_{n_1,n_2,\ldots,n_r} (with \sum_{i=1}^rn_i=n)$ by adding some edges within some $C_i$ for $i=1,\ldots,r$, in which there exists one $j$ such that $G^+[C_j]$ is disconnected, where $|C_j|=n_r$.
Then, $a(G^+)=a(G)$. 
\end{theorem}
\begin{proof}
    Assume, without loss of generality, that $j=r$, that is, $G^+[C_r]$ is disconnected. Suppose that $H_1,H_2,\ldots,H_s$ are the components of $G^+[C_r]$ with order $h_1,h_2,\ldots,h_s$, where  $\sum_{i=1}^sh_i=n_r$.
    Let $\L$ be the Laplacian matrix of $G$, and $\x=(\vec0_{n-n_r}^{\top},\frac{1}{h_1}\j_{h_1}^{\top},-\frac{1}{n_r-h_1}\j_{n_r-h_1}^{\top})^{\top}$ be a vector. Note that $\L\x=(n-n_r)\x$. It follows that $\x$ is a Fiedler vector of $G$ with $\x_{u}=\x_{v}=0$ for every pair of vertices $u,v\in C_1\cup C_2\cup \cdots\cup C_{r-1}$ and $\x_u=\x_v$ for every pair of vertices $u,v$ in $V(H_i)$, where $i=1,\ldots,s$. From this, together with Lemma \ref{le-adding_multi-edge-fiedlervector}, we obtain $a(G^+)=a(G)$.\qed
\end{proof}
Therefore, we have the following theorem concerning the algebraic connectivity of the token graph of a complete multipartite graph after adding some edges to it.

\begin{theorem}
    Let $G=K_{n_1,n_2,\ldots,n_r}$ be a complete multipartite graph with $n_1\le n_2\le \cdots\le n_r$ and $n=\sum_{i=1}^r n_i$. Let $G^+$ be a graph obtained from $G$ by adding some edges in 
    $C_i$, for some $i=1,\ldots,r$, in which there exists one $j$ such that $G^+[C_j]$ is disconnected, where $|C_j|=n_r$. Then,  $a(F_k(G^+))=a(G^+)=n-n_r$ for each integer $k$ with $1\le k\le \lfloor\frac{n}{2}\rfloor$.  
\end{theorem}
\begin{proof}
    It follows from Theorems \ref{th-al-com-mul} and \ref{th-token-com-mul} that $a(G^+)=a(G)=a(F_k(G))=n-n_r$. By Theorem \ref{th-add_edge}, it must be $a(F_k(G^+))=a(G^+)=n-n_r$.\qed
\end{proof}

\section{Some results on the {Laplacian} spectral radius of token graphs\label{sec:spactral_radius}}

{Consider $n\ge4$.} From {$\rho(F_k(G))\ge\rho(G)$ (see (\ref{eq-sp-rad-token}))}, it is natural to ask whether $\rho(F_k(G))=\rho(G)$. 
{Consider that $G$ is a graph having a vertex of degree $n-1$. 
Barik and Verma \cite{bv24} showed that $\rho(F_k(G))\ge\rho(G)$ with equality if and only if $G$ is the star graph.}
Therefore, we investigate what cases satisfy $\rho(F_k(G))=\rho(G)$ in order to answer the second question in the Introduction. Some known bounds related to the spectral radius are presented below.
\begin{lemma}[Grone and Merris \cite{gm_1994}, Zhang and Luo \cite{zl_2002}]
\label{le-lowerbound}
    Let $G$ be a simple connected graph with at least one edge, and maximum degree $\Delta(G)$. Then, $\rho(G)\ge\Delta(G)+1$ with equality if and only if there exists a vertex adjacent to all other vertices in $G$. 
\end{lemma}
\begin{lemma}[Anderson and Morley \cite{am_1985}]
\label{le-upperbound}
    Let $G$ be a simple graph. Then, \[\rho(G)\le \max\{d_u+d_v:uv\in E\}.\]
    {If $G$ is connected, then equality holds if and only if $G$ is semiregular bipartite.}
\end{lemma}
We first investigate the cases satisfying $\rho(F_k(G))=\rho(G)$, and we need the following result, which is related to the degrees of $G$.
\begin{theorem}
\label{th-maxdegree-notadjacent}
    {Let $G$ be a simple connected graph of order $n{(\ge4)}$ and $d_1,d_2,\ldots,d_n$ be its degree sequence (with $d_1\ge\cdots\ge d_n$). Let $u$ and $v$ be the two vertices with degree $d_1$ and $d_2$, respectively. If $u$ and $v$ are not adjacent, then $\rho(F_k(G))>\rho(G)$ for each integer $k$ with $2\le k\le \lfloor\frac{n}{2}\rfloor$.     }    
\end{theorem}
\begin{proof}
    Let $A=\{u,v\}$ be a vertex in $F_2(G)$. Note that $A_u=\{u,v_i\}$ and $A_v=\{u_j,v\}$ are adjacent to $A$, where $v_i\in N_G(v)$ and $u_j\in N_G(u)$. Then, $d_A=d_1+d_2$.
    Moreover, we have $\Delta(F_2(G))\ge d_A$. Combining with $\rho(G)\le d_1+d_2$ (it follows from Lemma \ref{le-upperbound}), 
    it must be 
    \begin{equation*}
        \rho(F_2(G))\ge \Delta(F_2(G))+1\ge d_A+1>d_1+d_2\ge \rho(G).
    \end{equation*}
    The first inequality holds from Lemma \ref{le-lowerbound}. Together with (\ref{eq-sp-rad-token}), it must be $\rho(F_k(G))\ge\cdots\ge\rho(F_2(G))>\rho(G)$.\qed
\end{proof}
Besides, with the following result, we give the relationship between the {Laplacian} spectral radius of a tree and that of its token graph.
\begin{theorem}\label{th:tree}
    Let $T$ be a tree {of order $n(\ge 4)$}. Then, {for each integer $k$ with $2\le k\le \lfloor\frac{n}{2}\rfloor$}, we have $\rho(F_k(T))\ge \rho(T)$ with equality if and only if $T\cong S_n$.
\end{theorem}
\begin{proof}
    Recall that $\rho(F_k(S_n))=\rho(S_n)=n$. It suffices to show that  $\rho(F_k(T))>\rho(T)$ if $T\ncong S_n$. Let $d_1,d_2,\ldots,d_n$ (with $d_1\ge d_2\ge\cdots\ge d_n$) be the degrees of vertices in $T$. Let $u$ and $v$ be the vertices with degrees $d_1$ and $d_2$, respectively. If $u$ and $v$ are not adjacent, from Theorem \ref{th-maxdegree-notadjacent}, the claim holds. Thus, we only consider $u\sim v$.
    Let 
    $\{u_1,u_2,\ldots,u_a\}=N_T(u)\backslash\{v\}$ and $\{v_1,v_2,\ldots,v_b\}=N_T(v)\backslash\{u\}${, where $a=d_1-1$ and $b=d_2-1$}.
    {Note that $T\ncong S_n$. }It follows from Lemma \ref{le-upperbound} that \begin{equation}
        \rho(T){<} d_1+d_2.\label{eq:spectralradiusoftree}
    \end{equation} 
    
    Let $A=\{u,v\}$ be a vertex in $F_2(T)$, 
    and $H$ be the induced subgraph of $T$ with vertex set $\{u,v,u_i,v_j\}$, where $i=1,\ldots,a$ and $j=1,\ldots,b$. It was shown in Barik and Verma \cite{bv24} that $F_2(H)$ is also an induced subgraph of $F_2(T)$. Moreover, the order of $F_2(H)$ is $\binom{a+b+2}{2}$ and it follows that $$\L(F_2(T))=\left(
    \begin{array}{cccc}
    \L(F_2(H))+\D & \B\\
    \B^{\top} & \C
    \end{array}
    \right),$$ where $\D$ is a diagonal matrix with at least one entry $0$ (since $A$ is not adjacent to any vertex in $V(F_2(T))\backslash V(F_2(H))$). Thus,
    \begin{equation}
        \rho(F_2(T))\ge\rho(\L(F_2(H))+\D)\ge \rho(F_2(H))+\lambda_1(\D)=\rho(F_2(H)).\label{eq:inducedtokengraph}
    \end{equation}
    The first inequality holds from Lemma \ref{le-submatrix} with $i=\binom{a+b+2}{2}$, and the second holds by Lemma \ref{le-matrix_add} also for the same value of $i$.
    
    Let $\pi=(C_1,C_2,\ldots,C_8)$ be a partition of $V(F_2(H))$ with $C_1=\left\{\{u,v\}\right\},$ $C_2=\left\{\{u_i,v\}:i=1,\ldots,a\right\},$ $C_3=\left\{\{u,v_j\}:j=1,\ldots,b\right\}$, $C_4=\left\{\{u,u_i\}:i=1,\ldots,a\right\}$, $C_5=\left\{\{v,v_j\}:i=1,\ldots,b\right\}$, $C_6=\left\{\{u_i,v_j\}:i=1,\ldots,a;j=1,\ldots,b\right\}$, $C_7=\{\{u_{i_1},u_{i_2}\}:$ $i_1,i_2=1,\ldots,a~(i_1\ne i_2)\}$, $C_8=\left\{\{v_{j_1},v_{j_2}\}:j_1,j_2=1,\ldots,b~(j_1\ne j_2)\right\}$.
    Note that $A_s\nsim A_t$ for every pair of vertices $A_s,A_t\in C_i$, where $i=1,\ldots,8$. Moreover, we have the quotient Laplacian matrix $\Q_L$ of $F_2(H)$ in relation to $\pi$ is  
    \[\Q_L=\left(
    \begin{array}{cccccccc}
    a+b & -a& -b&0&0&0&0&0\\
    -1 & b+2& 0&-1&0&-b&0&0\\
    -1 & 0& a+2&0&-1&-a&0&0\\
    0 & -1& 0&a&0&0&-a+1&0\\
    0 & 0& -1&0&b&0&0&-b+1\\
    0 & -1& -1&0&0&2&0&0\\
    0 & 0& 0&-2&0&0&2&0\\
    0 & 0& 0&0&-2&0&0&2
    \end{array}
    \right).\]
    Let $\Q'$ be {the} submatrix of $\Q_L$ induced by the first three rows and columns of $\Q_L$. It follows that $\Phi(\Q',x)=\left[x-(a+b+2)\right]\left[x^2-(a+b+2)x+(ab+a+b)\right]$, which implies that $\rho(\Q')\ge a+b+2$. Together with Lemma \ref{le-submatrix} with $i=3$, we obtain \begin{equation}
        \rho(\Q_L)\ge\rho(\Q')\ge a+b+2.\label{eq:tree3}
    \end{equation}
    Recall that $\pi$ is a regular partition. Then, the eigenvalues of $\Q_L$ must be the eigenvalues of $\L(F_2(H))$ with 
    \begin{equation}
        \rho(F_2(H))\ge\rho(\Q_L). \label{eq:submatrix}
    \end{equation}
    Thus, combining with (\ref{eq:spectralradiusoftree}), (\ref{eq:inducedtokengraph}), (\ref{eq:tree3}), and (\ref{eq:submatrix}), it must be $$\rho(F_2(T))\ge\rho(F_2(H)){\ge} a+b+2=d_1+d_2{>} \rho(T).$$
    Together with (\ref{eq-sp-rad-token}), we have $\rho(F_k(T))\ge\cdots\ge\rho(F_2(T))>\rho(T)$.\qed
\end{proof}
 \begin{figure}   
    \begin{center}
    \begin{tikzpicture}[scale=1.1,auto,swap]
    \vertex (1) at (0.7,0) [fill,label=above:1]{};
    \vertex (2) at (0,0.8) [fill,label=above:2]{};
    \vertex (3) at (0,-0.8) [fill,label=below:3]{};
    \vertex (4) at (1.7,0) [fill,label=above:4]{};
    \vertex (5) at (2.4,0.8) [fill,label=above:5]{};
    \vertex (6) at (2.4,-0.8) [fill,label=below:6]{};
    \draw[line width=0.6pt](1)--(2);
    \draw[line width=0.6pt](1)--(3);
    \draw[line width=0.6pt](1)--(4);
    \draw[line width=0.6pt](4)--(5);
    \draw[line width=0.6pt](4)--(6);

    \vertex (14) at (3.7,0) [fill,label=above:14]{};
    \vertex (24) at (5,2) [fill,label=above:24]{};
    \vertex (34) at (5,0.8) [fill,label=above:34]{};
    \vertex (15) at (5,-0.8) [fill,label=below:15]{};
    \vertex (16) at (5,-2) [fill,label=below:16]{};
    \vertex (12) at (6.5,2) [fill,label=above:12]{};
    \vertex (13) at (6.5,1.25) [fill,label=above:13]{};
    \vertex (23) at (8,1) [fill,label=above:23]{};
    \vertex (45) at (6.5,-1.25) [fill,label=below:45]{};
    \vertex (46) at (6.5,-2) [fill,label=below:46]{};
    \vertex (56) at (8,-1) [fill,label=below:56]{};
    \vertex (25) at (5.75,0.6) [fill,label=above:25]{};
    \vertex (35) at (5.75,-0.6) [fill,label=below:35]{};
    \vertex (26) at (7,0.6) [fill,label=above:26]{};
    \vertex (36) at (7,-0.6) [fill,label=below:36]{};
    \draw[line width=0.6pt](14)--(24);
    \draw[line width=0.6pt](14)--(34);
    \draw[line width=0.6pt](14)--(15);
    \draw[line width=0.6pt](14)--(16);
    \draw[line width=0.6pt](12)--(24);
    \draw[line width=0.6pt](13)--(34);
    \draw[line width=0.6pt](45)--(15);
    \draw[line width=0.6pt](46)--(16);
    \draw[line width=0.6pt](12)--(23);
    \draw[line width=0.6pt](13)--(23);
    \draw[line width=0.6pt](45)--(56);
    \draw[line width=0.6pt](46)--(56);
    \draw[line width=0.6pt](24)--(25);
    \draw[line width=0.6pt](24)--(26);
    \draw[line width=0.6pt](34)--(35);
    \draw[line width=0.6pt](34)--(36);
    \draw[line width=0.6pt](15)--(25);
    \draw[line width=0.6pt](16)--(26);
    \draw[line width=0.6pt](15)--(35);
    \draw[line width=0.6pt](16)--(36);
    \draw[] (9,0) circle (0.25);
    \node () at (9,-0.4) [label=above:{{\small $C_1$}}]{};
    \draw[] (10.2,1.4) circle (0.25);
    \node () at (10.2,1) [label=above:{{\small $C_2$}}]{};
    \draw[] (10.2,-1.4) circle (0.25);
    \node () at (10.2,-1.8) [label=above:{{\small $C_3$}}]{};
    \draw[] (12.5,1.4) circle (0.25);
    \node () at (12.5,1) [label=above:{{\small $C_4$}}]{};
    \draw[] (12.5,-1.4) circle (0.25);
    \node () at (12.5,-1.8) [label=above:{{\small $C_5$}}]{};
    \draw[] (11.5,0) circle (0.25);
    \node () at (11.5,-0.4) [label=above:{{\small $C_6$}}]{};
    \draw[] (13.5,0.6) circle (0.25);
    \node () at (13.5,0.2) [label=above:{{\small $C_7$}}]{};
    \draw[] (13.5,-0.6) circle (0.25);
    \node () at (13.5,-1) [label=above:{{\small $C_8$}}]{};
    \draw (9,0.25) edge [directed,line width=0.6pt,bend left] (9.95,1.35);
    \draw (9,-0.25) edge [directed,line width=0.6pt,bend right] (9.95,-1.35);
    \draw (10.2,1.15) edge [directed,line width=0.6pt,bend left] (9.22,0.1);
    \draw (10.2,-1.15) edge [directed,line width=0.6pt,bend right] (9.22,-0.1);
    \draw (10.4,1.55) edge [directed,line width=0.6pt,bend left] (12.3,1.55);
    \draw (12.3,1.25) edge [directed,line width=0.6pt,bend left] (10.45,1.35);
    \draw (10.4,-1.55) edge [directed,line width=0.6pt,bend right] (12.3,-1.55);
    \draw (12.3,-1.25) edge [directed,line width=0.6pt,bend right] (10.45,-1.35);
    \draw (12.75,1.45) edge [directed,line width=0.6pt,bend left] (13.5,0.85);
    \draw (13.25,0.65) edge [directed,line width=0.6pt,bend left] (12.65,1.2);
    \draw (12.75,-1.45) edge [directed,line width=0.6pt,bend right] (13.5,-0.85);
    \draw (13.25,-0.65) edge [directed,line width=0.6pt,bend right] (12.65,-1.2);
    \draw (10.4,1.25) edge [directed,line width=0.6pt,bend left] (11.5,0.25); 
    \draw (11.3,0.15) edge [directed,line width=0.6pt,bend left] (10.27,1.15); 
    \draw (10.4,-1.25) edge [directed,line width=0.6pt,bend right] (11.5,-0.25); 
    \draw (11.3,-0.15) edge [directed,line width=0.6pt,bend right] (10.27,-1.15); 
    \node () at (9.3,0.9) [label=above:{\small 2}]{};
    \node () at (9.5,0.7) [label=right:{\small 1}]{};
    \node () at (9.3,-0.9) [label=below:{\small 2}]{};
    \node () at (10.1,-0.7) [label=left:{\small 1}]{};
    \node () at (11.5,1.7) [label=above:{\small 1}]{};
    \node () at (11.5,0.85) [label=above:{\small 1}]{};
    \node () at (11.5,-1.7) [label=below:{\small 1}]{};
    \node () at (11.5,-0.85) [label=below:{\small 1}]{};
    \node () at (13.1,1.3) [label=right:{\small 1}]{};
    \node () at (12.9,0.95) [label=below:{\small 2}]{};
    \node () at (13.1,-1.3) [label=right:{\small 1}]{};
    \node () at (12.9,-0.95) [label=above:{\small 2}]{};
    \node () at (11.5,-1.7) [label=below:{\small 1}]{};
    \node () at (11.5,-0.85) [label=below:{\small 1}]{};
    \node () at (11.2,0.8) [label=left:{\small 2}]{};
    \node () at (10.75,0.35) [label=left:{\small 1}]{};
    \node () at (11.2,-0.8) [label=left:{\small 2}]{};
    \node () at (10.75,-0.35) [label=left:{\small 1}]{};
    \end{tikzpicture}
    \caption{The graph $H$ (left), its $2$-token graph $F_2(H)$ (middle), and the quotient graph of a regular partition of $F_2(H)$ (right).}
     \label{fig:tree}
    \end{center}
    \end{figure}
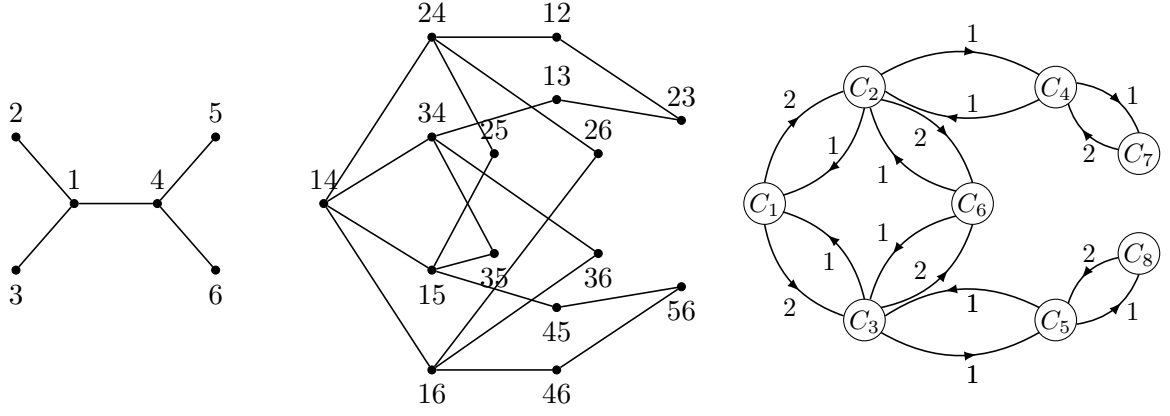

Let us show an example.
    
\begin{example}
    {Let $H$ be the graph depicted in Figure \ref{fig:tree}}. Its $2$-token graph $F_2(H)$ has a regular partition $\pi$ with cells $C_1=\{14\}$, $C_2=\{24,34\}$, $C_3=\{15,16\}$, $C_4=\{12,13\}$, $C_5=\{45,46\}$, $C_6=\{25,26,35,36\}$, $C_7=\{23\}$, $C_8=\{56\}$, which are the vertices of the quotient graph $H/\pi$ of $F_2(H)$, and with quotient Laplacian matrix
    
$$
\Q_L=
\left(
\begin{array}{rrrrrrrrr}
4 & -2 & -2 & 0 & 0 & 0 & 0 & 0\\
-1 & 4 & 0 & -1 & 0 & -2 & 0 & 0\\
-1 & 0 & 4 & 0 & -1 & -2 & 0 & 0\\
0 & -1 & 0 & 2 & 0 & 0 & -1 & 0\\
0 & 0 & -1 & 0 & 2 & 0 & 0 & -1\\
0 & -1 & -1 & 0 & 0 & 2 & 0 & 0\\
0 & 0 & 0 & -2 & 0 & 0 & 2 & 0\\
0 & 0 & 0 & 0 & -2 & 0 & 0 & 2
\end{array}
\right),
$$
with $\rho(F_2(H))=6.6262=\rho(\Q_L)>6=\rho(\Q')$ and $\rho(H)=4.5616<6$, where $$\Q'=\left(
\begin{array}{rrrr}
4 & -2 & -2 \\
-1 & 4 & 0\\
-1 & 0 & 4
\end{array}
\right).$$
\end{example}


Next, we consider regular graphs. 
\begin{theorem}
\label{th-regular}
    Let $G$ be a connected $r$-regular graph {of order} $n${($\ge 4$)}. Then, {for each integer $k$ with $2\le k\le \lfloor\frac{n}{2}\rfloor$}, we have $\rho(F_k(G)){>} \rho(G)$. 
\end{theorem}
\begin{proof}
    For $r=1$, the case is trivial. Suppose that $2\le r\le n-1$. First, consider $r=n-1$, that is, $G$ is a complete graph. Note that the eigenvalues of $\L(F_k(G))$ are $\left\{j(n-j+1):0\le j\le k\right\}$ (see Theorem \ref{le-k-algebraic}), and the distinct eigenvalues of $\L(G)$ are $0,n$. 
    {It follows that $\rho(F_k(K_n))>\rho(K_n)$ for $n\ge 4$} (since $j\le k\le \lfloor\frac{n}{2}\rfloor$).
    Next, assume that $G$ is not a complete graph. Thus, there must exist a pair of non-adjacent vertices with the same degree $r$ in $G$ for $2\le r\le n-2$. By Theorem \ref{th-maxdegree-notadjacent}, the result holds. \qed
\end{proof}

Let us show an example of the last theorem.

\begin{example}
    Let $G$ be {the} cubic graph {depicted in Figure \ref{fig:cubic}. Its $2$-token graph $F_2(G)$ satisfies $\rho(F_2(G))=8.5616>5=\rho(G)$}.
    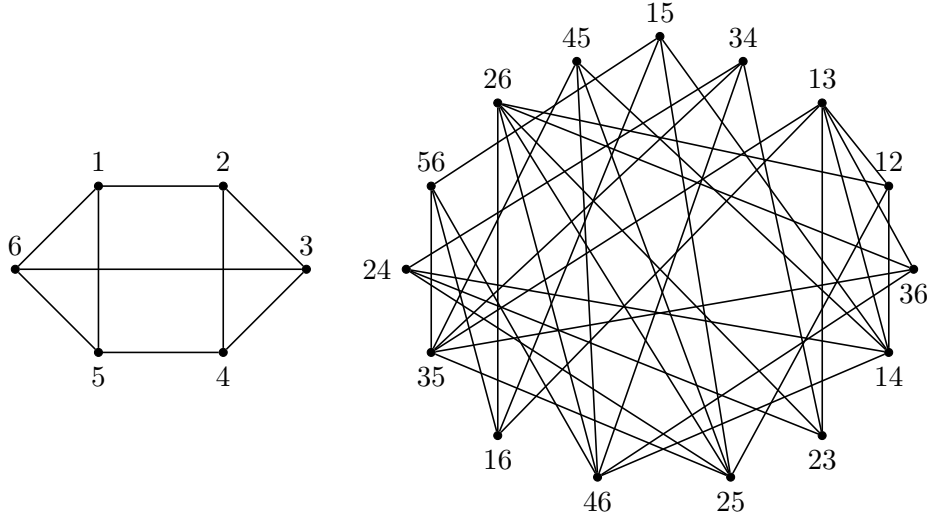
\begin{figure}   
    \begin{center}
    \begin{tikzpicture}[scale=1.1,auto,swap]
    \vertex (1) at (0,1) [fill,label=above:1]{};
    \vertex (2) at (1.5,1) [fill,label=above:2]{};
    \vertex (5) at (0,-1) [fill,label=below:5]{};
    \vertex (4) at (1.5,-1) [fill,label=below:4]{};
    \vertex (3) at (2.5,0) [fill,label=above:3]{};
    \vertex (6) at (-1,0) [fill,label=above:6]{};
    \draw[line width=0.6pt](1)--(2);
    \draw[line width=0.6pt](3)--(2);
    \draw[line width=0.6pt](4)--(3);
    \draw[line width=0.6pt](5)--(4);
    \draw[line width=0.6pt](6)--(5);
    \draw[line width=0.6pt](1)--(6);
    \draw[line width=0.6pt](1)--(5);
    \draw[line width=0.6pt](6)--(3);
    \draw[line width=0.6pt](4)--(2);
    \vertex (06) at (9.8,0) [fill,label=below:36]{};
    \vertex (01) at (8.7,2) [fill,label=above:13]{};
    \vertex (011) at (5.75,2.5) [fill,label=above:45]{};
    \vertex (012) at (3.7,0) [fill,label=left:24]{};
    \vertex (05) at (4,-1) [fill,label=below:35]{};
    \vertex (02) at (7.6,-2.5) [fill,label=below:25]{};
    \vertex (09) at (4.8,-2) [fill,label=below:16]{};
    \vertex (03) at (6,-2.5) [fill,label=below:46]{};
    \vertex (08) at (9.5,-1) [fill,label=below:14]{};
    \vertex (015) at (7.75,2.5) [fill,label=above:34]{};
    \vertex (014) at (4,1) [fill,label=above:56]{};
    \vertex (010) at (6.75,2.8) [fill,label=above:15]{};
    \vertex (013) at (4.8,2) [fill,label=above:26]{};
    \vertex (07) at (9.5,1) [fill,label=above:12]{};
    \vertex (04) at (8.7,-2) [fill,label=below:23]{};
    \draw[line width=0.6pt](02)--(010);
    \draw[line width=0.6pt](02)--(011);
    \draw[line width=0.6pt](02)--(012);
    \draw[line width=0.6pt](02)--(05);
    \draw[line width=0.6pt](02)--(013);
    \draw[line width=0.6pt](02)--(07);
    \draw[line width=0.6pt](09)--(010);
    \draw[line width=0.6pt](09)--(014);
    \draw[line width=0.6pt](09)--(013);
    \draw[line width=0.6pt](03)--(011);
    \draw[line width=0.6pt](03)--(06);
    \draw[line width=0.6pt](03)--(014);
    \draw[line width=0.6pt](03)--(08);
    \draw[line width=0.6pt](03)--(015);
    \draw[line width=0.6pt](03)--(013);
    \draw[line width=0.6pt](07)--(01);
    \draw[line width=0.6pt](07)--(08);
    \draw[line width=0.6pt](07)--(013);
    \draw[line width=0.6pt](014)--(010);
    \draw[line width=0.6pt](014)--(05);
    \draw[line width=0.6pt](04)--(01);
    \draw[line width=0.6pt](04)--(013);
    \draw[line width=0.6pt](04)--(012);
    \draw[line width=0.6pt](04)--(015);
    \draw[line width=0.6pt](08)--(010);
    \draw[line width=0.6pt](08)--(01);
    \draw[line width=0.6pt](08)--(011);
    \draw[line width=0.6pt](08)--(012);
    \draw[line width=0.6pt](015)--(012);
    \draw[line width=0.6pt](013)--(06);
    \draw[line width=0.6pt](05)--(01);
    \draw[line width=0.6pt](05)--(011);
    \draw[line width=0.6pt](05)--(015);
    \draw[line width=0.6pt](05)--(06);
    \draw[line width=0.6pt](06)--(01);
    \draw[line width=0.6pt](09)--(01);
    \end{tikzpicture}
    \caption{A cubic graph $G$ (left) and its $2$-token graph $F_2(G)$ (right).}
     \label{fig:cubic}
    \end{center}
    \end{figure}
\end{example}
Next, we consider what cases satisfy $\rho(F_k(G))>\rho(G)$.
\begin{theorem}
\label{th-degree-spectralradius}
    {Suppose $G$ is a connected graph on $n{(\ge4)}$ vertices containing two non-adjacent vertices $u$ and $v$ such that $d_u+d_v\ge n-1$. Then, for each integer $k$ with $2\le k\le \lfloor\frac{n}{2}\rfloor$ we have $\rho(F_k(G))>\rho(G)$. }
\end{theorem}
\begin{proof}
    {Consider the vertex $A=\{u,v\}$ in $F_2(G)$}. Note that $A_u=\{{v},u_i\}$ and $A_v=\{{u,v_j}\}$ are adjacent to $A$, where $u_i\in N_G(u)$ and $v_j\in N_G(v)$. 
    Then, $d_A= d_u+d_v\ge n-1$. Moreover, $\Delta(F_2(G))\ge d_A$.
    It follows from Lemma \ref{le-lowerbound} that \[\rho(F_2(G))>\Delta(F_2(G))+1\ge n-1+1=n.\] 
    The first inequality holds because there is no vertex in $F_2(G)$ that is adjacent to all other vertices in $F_2(G)$. {(For example, choose a vertex $B=\{w,z\}$ with $w,z\notin \{u,v\}$. Then $A$ and $B$ are non-adjacent in $F_2(G)$.)} 
    Recall that $\rho(G)\le \rho(K_n)=n$ for any graph $G$. It follows that $\rho(F_2(G))>n\ge \rho(G)$. Together with \eqref{eq-sp-rad-token}, it must be $\rho(F_k(G))\ge\cdots\ge\rho(F_2(G))>\rho(G).$\qed
\end{proof}
\begin{corollary}
    Let $G$ be a connected graph of order $n{(\ge4)}$ and maximum degree $n-2$. Then, $\rho(F_k(G))>\rho(G)$ for each $k=2,\ldots,\lfloor\frac{n}{2}\rfloor$. 
\end{corollary}
\begin{proof}
    Let $u$ be the vertex of maximum degree, and $v$ be a vertex that is not adjacent to $u$. Note that $d_u+d_v\ge n-2+1=n-1$. It follows from Theorem \ref{th-degree-spectralradius} that the result holds. \qed
\end{proof}
Based on those conditions above, we show $\rho(F_k(G))>\rho(G)$ when $G$ is a graph in the following classes.
\begin{corollary}
    {Let $G$ be a graph satisfying one of the following:
    \begin{itemize}
        \item [$(i)$] $G$ is a graph containing a cut clique $K_r$, for some $r\ge \frac{n-1}{2}$, such that every vertex in $K_r$ is adjacent to all vertices in $G-K_r$.
        \item[$(ii)$] $G=K_{n_1,n_2,\ldots,n_r}$ is a complete multipartite graph with parts $C_1,C_2,\ldots,C_r$ of sizes $|C_i|=n_i$, where $n_1\le n_2\le\cdots\le n_r$, $r\ge2$, and $2\le n_j\le \lfloor\frac{n}{2}\rfloor$ for some $1\le j\le r$.
    \end{itemize}
    For each $k\in \{2,\ldots,\lfloor\frac{n}{2}\rfloor\}$, we have $\rho(F_k(G))>\rho(G)$.}
\end{corollary}
\begin{proof}
    {By} Theorem \ref{th-degree-spectralradius}, it suffices to {find two non-adjacent vertices $u,v\in V(G)$ such that $d_u+d_v\ge n-1$}. \\
    $(i)$ {Choose} $u$ and $v$ lying in distinct components of $G-K_r$. 
    Then $u$ and $v$ are non-adjacent, and moreover, as $d_u\ge r$ and $d_v\ge r$, we obtain $d_u+d_v\ge 2r\ge n-1$. \\
    $(ii)$ {Choose} $u,v\in C_j$, {so $u$ and $v$ are non-adjacent, and $d_u=d_v=n-n_j$. Hence, $d_u+d_v=2(n-n_j)\ge 2(n-\lfloor\frac{n}{2}\rfloor)>n-1$.} \qed
\end{proof}
\begin{remark}
    For a complete bipartite graph $G=K_{n_1,n_2}$ with $n_1+n_2=n$, which is not isomorphic to a star, we have $\rho(F_k(G))>\rho(G)$ for each $k=2,\ldots,\lfloor\frac{n}{2}\rfloor$.
\end{remark}
\begin{corollary}
    Let $G(\ncong S_n)$ be an $(s,t)$-semiregular bipartite graph. Then, $\rho(F_k(G))>\rho(G)$ for each $k=2,\ldots,\lfloor\frac{n}{2}\rfloor$. 
\end{corollary}
\begin{proof}
    Without loss of generality, suppose that $s\ge t$. Let $(X,Y)$ be a bipartition of $G$ with $d_u=s$ for any vertex $u\in X$. From Theorem \ref{th-maxdegree-notadjacent}, the result holds. \qed
\end{proof}

Moreover, we give a lower bound of the Laplacian spectral radius of the token graph{s} of complete bipartite graph{s}. 
Let $\M\in  \mathbb{R}^{(k+1)\times(k+1)}$ be a matrix with the first row $(kn_2,-kn_2,\vec0_{k-1}^{\top})$, the $(k+1)$-th row $(\vec0_{k-1}^{\top},-kn_1,kn_1)$, and the $(i+1)$-th row $$\left(\vec0_{i-1}^{\top},-i[n_1-(k-i)],i[n_1-(k-i)]+(k-i)(n_2-i),-(k-i)(n_2-i),\vec0_{k-i-1}^{\top}\right),$$ where $i=1,\ldots,k-1$. A partition of $\M$ is 
\begin{equation}\label{eq:partitionM}
  \left(
\begin{array}{cccc}
\M_{11} & \M_{12}\\
\M_{21} & \M_{22}
\end{array}
\right),  
\end{equation}
with $\M_{11}\in\Re^{(k-n_1)\times (k-n_1)}$, and $\M_{22}\in\Re^{(n_1+1)\times (n_1+1)}$. 
Suppose first that $n_1,n_2\ge k$. In this case, let $\Q^*_L=\M$.
Assume next that $n_1<k\le n_2$. Now, let $\Q^*_L$ be the principal submatrix $
\M_{22}$ of $\M$.
\begin{theorem}
\label{th-spectralradius-bipartion}
    Let $G=K_{n_1,n_2}$ be a complete bipartite graph with $n=n_1+n_2$. Then,  $\rho(F_k(G))\ge \rho(\Q^*_L)$ for $k=1,\ldots,\lfloor\frac{n}{2}\rfloor$.
\end{theorem}
\begin{proof}
    For $k=1$, the case is trivial. Consider $k\ge2$.
    Let $(X,Y)$ be the bipartition of $K_{n_1,n_2}$ with $X=\{u_1,u_2,\ldots,u_{n_1}\}$ and $Y=\{v_1,v_2,\ldots,v_{n_2}\}$. 
    Suppose first that $n_1,n_2\ge k$. Let \[C_i=\left\{\{u_{s_1},u_{s_2},\ldots,u_{s_{k+1-i}},v_{t_1},\ldots,v_{t_{i-1}}\}:u_{s_j}\in X,v_{t_j}\in Y\right\}\]
    for $i=1,\ldots,k+1${, where $C_i$ represents the set of $k$-subsets with exactly $i-1$ vertices in $Y$ and $k+1-i$ in $X$.}
    Note that $A\nsim A'$ for every pair of vertices $A,A'\in C_i$, where $i=1,\ldots,k+1$. Let $A=\{u_{s_1},u_{s_2},\ldots,u_{s_{k+1-i}},v_{t_1},\ldots,v_{t_{i-1}}\}\in C_i$. 
    {For $|i-j|\ge2$, the symmetric difference between $A$ and any vertex $B$ in $C_{j}$ has size at least $4$,} 
    it shows that $A$ is not adjacent to any vertex {$B$} in $C_{j}$.
    Let $A^*=(A\backslash\{v\})\cup\{w\}\in C_{i-1}$, where $v\in \{v_{t_1},\ldots,v_{t_{i-1}}\}$ and $w\in X\backslash\{u_{s_1},u_{s_2},\ldots,u_{s_{k+1-i}}\}$. Note that $A\sim A^*$ (since $v\sim w$ in $G$). Moreover, the number of the neighbors of $A$ in $C_{i-1}$ is $(i-1)(n_1-(k+1-i))$. Similarly, the number of the neighbors of $A$ in $C_{i+1}$ is $(k+1-i)(n_2-(i-1))$. Thus, the quotient Laplacian matrix of $F_k(G)$ in relation to $\pi=(C_1,C_2,\ldots,C_{k+1})$ is $\Q^*_L$ {(see (\ref{eq:partitionM}))}, which implies that $\rho(\Q^*_L)$ is an eigenvalue of $\L(F_k(G))$. Therefore, it must be $\rho(F_k(G))\ge \rho(\Q^*_L)$. 
    
    For the remaining case $n_1<k\le n_2$~(it must be $n_2\ge k$ since $n_2\ge \lfloor\frac{n}{2}\rfloor$), {we find that the sets $C_1,C_2,\ldots,C_{k-n_1}$ are empty. Moreover, the adjacency relationships among $C_{k-n_1+1},\ldots,C_{k+1}$ are the same as described above. Therefore, with respect to the partition $\pi'=(C_{k-n_1+1},\ldots,C_{k+1})$, the quotient Laplacian matrix of $F_k(G)$ is $\M_{22}$, see (\ref{eq:partitionM}).
    Thus,} we get $\rho(F_k(G))\ge \rho(\Q^*_L)$.
    \qed
\end{proof}
\begin{corollary}
\label{coro-bipartite}
     Let $G=K_{n_1,n_2}$ be a complete bipartite graph with $n=n_1+n_2$. Then, $\rho(F_2(G))\ge 2(n-1)$.
\end{corollary}
\begin{proof}
    Note that $$\Q^*_L=
\left(
\begin{array}{cccc}
2n_2 & -2n_2& 0\\
-(n_1-1) & n-2& -(n_2-1)\\
0& -2n_1& 2n_1
\end{array}
\right).$$ 
{Then,} $\rho(\Q^*_L)=2(n-1)$. Thus, from Theorem \ref{th-spectralradius-bipartion}, it must be $\rho(F_2(G))\ge 2(n-1)$. \qed
\end{proof}
\begin{remark}
    The bound in Theorem \ref{th-spectralradius-bipartion} is tight. For $k=2$ and $G=K_{\frac{n}{2},\frac{n}{2}}$, from Corollary \ref{coro-bipartite}, we have $\rho(F_2(G))\ge 2n-2$. Moreover, Barik and Verma \cite{bv24} also showed that $F_2(G)$ is an $(n,n-2)$-semiregular bipartite graph, then it follows from Lemma \ref{le-upperbound} that $\rho(F_2(G))\le 2n-2$. It implies that $\rho(F_2(G))=\rho(\Q^*_L)=2(n-1)$ in this case. 
\end{remark}

\end{document}